\documentclass[11pt]{article}
\usepackage[a4paper]{anysize}\marginsize{3.5cm}{3.5cm}{1.3cm}{2cm}
\pdfpagewidth=\paperwidth \pdfpageheight=\paperheight
\usepackage{amsfonts,amssymb,amsthm,amsmath,eucal}
\usepackage{pgf}
\usepackage{bbm,array}
\usepackage{tikz} 
\usepackage{float}
\restylefloat{table}
\usepackage{subfigure}
\usepackage{caption}
\usepackage{enumerate}

\usetikzlibrary{arrows}
\theoremstyle{plain}
\newtheorem{thm}{Theorem}[section]
\newtheorem{theorem}[thm]{Theorem}

\newtheorem{lemma}[thm]{Lemma}
\newtheorem{proposition}[thm]{Proposition}
\newtheorem{corollary}[thm]{Corollary}
\newtheorem{conjecture}[thm]{Conjecture}

\theoremstyle{definition}
\newtheorem{definition}[thm]{Definition}
\newtheorem{remark}[thm]{Remark}
\newtheorem{example}[thm]{Example}

\newcommand\beginproof[1]{\trivlist\item[\hskip\labelsep{\em #1.}]}

\newcommand\proofof[1]{\beginproof{Proof of #1}}
\def\endproof{\hspace*{\fill}\endproofsymbol\endtrivlist}

\def\endproofsymbol{\frame{\rule[0pt]{0pt}{6pt}\rule[0pt]{6pt}{0pt}}}

\newtheorem{thevarthm}[thm]{\varthmname}

\newenvironment{varthm*}[1]{\trivlist\item[]{\bf #1.}\it}{\endtrivlist}

\renewcommand\geq{\geqslant}

\renewcommand\leq{\leqslant}

\newcommand\be{\begin{eqnarray*}}
\newcommand\ee{\end{eqnarray*}}

\newcommand\K{\mathbb K}

\newcommand\C{\mathbb C}
\newcommand\Z{\mathbb Z}

\newcommand\F{\mathbb F}

\renewcommand\P{\mathbb P}

\newcommand\call{{\mathcal L}}

\newcommand\calp{{\mathcal P}}
\newcommand\calt{{\mathcal T}}

\newcommand\newop[2]{\def#1{\mathop{\rm #2}\nolimits}}
\newop\log{log}
\newop\ord{ord}
\newop\Gal{Gal}
\newop\SL{SL}
\newop\GL{GL}
\newop\Bl{Bl}
\newop\mult{mult}
\newop\mass{mass}
\newop\div{div}
\newop\codim{codim}
\newop\sing{sing}
\newop\vdim{vdim}
\newop\edim{edim}
\newop\Ass{Ass}
\newop\size{size}
\newop\reg{reg}
\newop\areg{areg}
\newop\asreg{asreg}
\newop\satdeg{satdeg}
\newop\supp{supp}
\newop\gin{gin}
\newop\ini{in}
\newop\vol{vol}
\newop\sat{sat}
\newop\length{length}
\newop\depth{depth}
\newop\characteristic{char}

\def\keywordname{{\bfseries Keywords}}%
\def\keywords#1{\par\addvspace\medskipamount{\rightskip=0pt plus1cm
\def\and{\ifhmode\unskip\nobreak\fi\ $\cdot$
}\noindent\keywordname\enspace\ignorespaces#1\par}}
\def\subclassname{{\bfseries Mathematics Subject Classification
(2000)}\enspace}
\def\subclass#1{\par\addvspace\medskipamount{\rightskip=0pt plus1cm
\def\and{\ifhmode\unskip\nobreak\fi\ $\cdot$
}\noindent\subclassname\ignorespaces#1\par}}

\definecolor{qqqqff}{rgb}{0,0,0}
\definecolor{uuuuuu}{rgb}{0,0,0}
\definecolor{zzttqq}{rgb}{0,0,0}
\definecolor{xdxdff}{rgb}{0,0,0}
\definecolor{ttqqqq}{rgb}{0.2,0.,0.}
\definecolor{uuuuuu}{rgb}{0.26666666666666666,0.26666666666666666,0.26666666666666666}
\definecolor{xdxdff}{rgb}{0.49019607843137253,0.49019607843137253,1.}
\definecolor{wwccqq}{rgb}{0.4,0.8,0.}
\definecolor{qqqqcc}{rgb}{0.,0.,0.8}
\definecolor{ffttww}{rgb}{1.,0.2,0.4}

\begin{document}

\author{M.~Dumnicki, D.~Harrer, J.~Szpond\footnote{The first and the third authors were partially supported by National Science Centre, Poland, grant 2014/15/B/ST1/02197}}
\title{On absolute linear Harbourne constants}
\date{\today}
\maketitle
\thispagestyle{empty}

\begin{abstract}
In the present note we study absolute linear Harbourne constants. These are invariants which were introduced in \cite{IMRN}
in order to relate the lower bounds on the selfintersection of negative curves on birationally equivalent surfaces to the complexity
of the birational map between them. We provide various lower and upper bounds on Harbourne constants and give their values
for the number of lines $s$ of the form $p^{2r}+p^r+1$ for any prime number $p$ and also for all values of $s$ up to $31$. This
extends considerably results of the third author obtained earlier in  \cite{Szp2015}.

\subclass{14C20; 52C30; 05B30}
\end{abstract}


\section{Introduction}
\label{intro} Arrangements of lines were  introduced to algebraic
geometry by Hirzebruch in his papers concerning the geography of
surfaces (i.e. construction of surfaces $X$ with prefixed invariants
$c_1^2(X)$ and $c_2(X)$), see \cite{Hir83}, \cite{BHH87}.

Multiplier ideals defined by arrangements of lines were studied by Teitler \cite{Tei07} and Musta\c t\u a \cite{Mus2006}.

Recently arrangements of lines appeared in the ideas revolving
around the Boun\-ded Negativity Conjecture (BNC for short), see
\cite{BNC} for the background of the Conjecture and \cite{IMRN},
\cite{Pok2015} for the role of configurations of lines. Whereas BNC
is relevant only over a field of characteristic zero, some related
problems are of interest over arbitrary fields. In \cite{IMRN} the
authors introduced and began to study linear Harbourne constants.
These are certain invariants computed by configurations of lines in
the projective plane. Even though the Bounded Negativity fails in
positive characteristic, it is clear from Definition \ref{de:
Harbourne} that for a fixed $d$, the linear Harbourne constant
$H(d)$ is a finite number (because the number of combinatorial possibilities
for invariants of a configuration of $d$ lines is finite). It is
interesting to estimate these numbers because in particular they
measure the discrepancy between combinatorial data sets, see
\cite{BruRys49} and those sets which come from geometric
configurations defined over some fields.

For the purpose of this note, a configuration $\call$ is a finite
set of mutually distinct lines $\call=\{L_1,\ldots, L_d\}$. Given a
configuration $\call$, we define its singular set
$\calp(\call)=\{P_1,\ldots,P_s\}$ as a set of points where two or
more lines intersect. This is the same as the singular locus of the
divisor $L_1+\ldots +L_d$. For a point $P\in \calp (\call)$, we
denote by $m_{\call}(P)$ its \emph{multiplicity}, i.e. the number of
lines which pass through $P$.  We have the following definitions.
\begin{definition}\label{de: Harbourne}
The \emph{linear Harbourne constant of a configuration of lines
$\call$ in the projective plane $\P^2(\K)$}  is the rational number
  \begin{equation}\label{l-H-constant}
  H(\K,\call)=\frac{d^2-\sum_{k=1}^s m_{\call}(P_k)^2}{s}.
  \end{equation}
  The \emph{linear Harbourne constant of $d$ lines over $\K$} is defined as the minimum
  \[
  H(\K,d):=\min H(\K,\call)
  \]
  taken over all configurations $\call$ of $d$ lines.\\
  Finally the \emph{absolute linear Harbourne constant of $d$ lines} is the minimum
  \[
  H(d):=\min_{\K} H(\K,d)
  \]
  taken over all fields $\K$.
\end{definition}

In order to alleviate the notation we define first the set
\[
Q=\{ q=p^r, \quad {\text{$p$ is prime}}, \quad r\in \Z_{>0}\}.
\]
For an integer $d$, we define $q(d)$ as the least number $q\in Q$ satisfying
\[
d\leq q^2+q+1
\]
and $r(d)$ as the largest number $r\in Q$ satisfying
\[
r^2+r+1\leq d.
\]
Systematic investigation of absolute linear Harbourne constants
$H(d)$ was initiated in \cite{Szp2015}. Results stated there and
computer supported experiments have led us to formulate the
following conjecture.
\begin{conjecture}\label{con: Hconstants}
For $d\geq 2$ let $q=q(d)$ and let $i:=q^2+q+1-d$.

If $i\leq 2q-2$, then
\[
H(d)=h(d)
\]
where
$$h(d)=\frac{q^2+q+1-i-\varepsilon_1(i)m_1(i)-\varepsilon_2(i)m_2(i)-t_{q-1}(i)(q-1)-t_q(i)
q-t_{q+1}(i)(q+1)}{\varepsilon_1(i)+\varepsilon_2(i)+t_{q-1}(i)+t_q(i)+t_{q+1}(i)},$$
with
$$m_1(i)=q+1-i, \quad \quad m_2(i)=2q+1-i$$
   \begin{equation*}
   \varepsilon_1(i)=
   \left\{\begin{array}{lcc}
   1 &   & \text{for $0\leq i\leq q-1$}\\
   0 &   & \text{otherwise}
   \end{array}\right.,   \quad \quad
      \varepsilon_2(i)=
   \left\{\begin{array}{lcc}
   1 &   & \text{for $i> q+1$}\\
   0 &   & \text{otherwise}
   \end{array}\right.,
   \end{equation*}

      \begin{equation*}
   t_{q-1}(i)=
   \left\{\begin{array}{lcc}
   qi-q^2-q &   & \text{for $i> q+1$}\\
   0 &   & \text{otherwise}
   \end{array}\right.,
   \end{equation*}
   \begin{equation*}
   t_q(i)=
   \left\{\begin{array}{lcc}
   qi &   & \text{for $i\leq q+1$}\\
   2q^2-(i-2)q-1 &   & \text{for $i>q+1$}
   \end{array}\right.,
   \end{equation*}
      \begin{equation*}
   t_{q+1}(i)=
   \left\{\begin{array}{lcc}
   q^2+q-iq &   & \text{for $i\leq q+1$}\\
   0 &   & \text{otherwise}
   \end{array}\right..
   \end{equation*}
Moreover for $i=2q-1$ we have
\[
H(d)=-\frac{q^3-q^2+2q-2}{q^2+q-1}.
\]
\end{conjecture}
\begin{remark}
We do not know what happens for $d$ such that $d\leq (q(d)-1)^2+q(d)$. The first such $d$ is $d=32$ with $q(32)=7$. See the end of the last section.
\end{remark}
This conjecture has been verified for $d\leq 10$ in \cite{Szp2015}. In the present paper, we extend the range of the validity of the Conjecture to $d\leq 31$. This is our first main result. We repeat the results from \cite{Szp2015} for completeness.
\begin{theorem}[Values of absolute linear Harbourne constants]\label{thm values}
For $2\geq d\geq 31$ we have

   \begin{table}[H]
   \centering
   \renewcommand{\arraystretch}{1.2}
\begin{minipage}{.32\textwidth}
        \centering   
\begin{tabular}[t]{|c|c|}
  \hline
     $d$ & $H(d)$  \\
  \hline
     $2$ & $0$ \\
  \hline
     $3$ & $-1$\\
  \hline
     $4$ & $-4/3\;\approx\; -1.333$ \\
  \hline
     $5$ & $-3/2\;=\;-1.5$ \\
  \hline
     $6$ & $-12/7\;\approx\; -1.714$ \\
  \hline
     $7$ & $-2$\\
  \hline
  \hline
     $8$ & $-2$\\
  \hline
     $9$ & $-9/4\;=\;-2.25$ \\
  \hline
     $10$ & $-29/12\;\approx\; -2.416$ \\
  \hline
     $11$ & $-33/13\;\approx\; -2.538$ \\
  \hline
     $12$ & $-36/13\;\approx\; -2.769$ \\
  \hline
     $13$ & $-3$ \\
  \hline
  \hline
\end{tabular}
\end{minipage}
\begin{minipage}{.32\textwidth}
        \centering
\begin{tabular}[t]{|c|c|}
  \hline
     $d$ & $H(d)$  \\
  \hline
     $14$ & $-54/19\;\approx\; -2.842$ \\
  \hline
     $15$ & $-3$ \\
  \hline
     $16$ & $-16/5\;=\;-3.2$ \\
  \hline
     $17$ & $-67/20\;=\;-3.35$ \\
  \hline
     $18$ & $-24/7\;\approx\; -3.428$ \\
  \hline
     $19$ & $-76/21\;\approx\; -3.619$\\
  \hline
     $20$ & $-80/21\;\approx\; -3.809$ \\
  \hline
     $21$ & $-4$ \\
  \hline
  \hline
\end{tabular}
\end{minipage}
\begin{minipage}{.32\textwidth}
        \centering
\begin{tabular}[t]{|c|c|}
  \hline
     $d$ & $H(d)$ \\
  \hline
     $22$ & $-108/29\;\approx\; -3.724$ \\
  \hline
     $23$ & $-115/30\;\approx\; -3.833$\\
  \hline
     $24$ & $-4$ \\
  \hline
     $25$ & $-125/30\;\approx\; -4.166$ \\
  \hline
     $26$ & $-129/30\;=\;4.3$ \\
  \hline
     $27$ & $-135/31\;\approx\; -4.354$\\
  \hline
     $28$ & $-140/31\;\approx\; -4.516$ \\
  \hline
     $29$ & $-145/31\;\approx\; -4.677$ \\
  \hline
     $30$ & $-150/31\;\approx\; -4.838$ \\
  \hline
     $31$ & $-5$ \\
  \hline
  \hline
\end{tabular}
\end{minipage}
   \caption{$ $  Values of $H(d)$ for up to $31$ lines}
   \label{tab: values}
   \end{table}
\end{theorem}

\begin{remark}
It comes as a surprise that the function $H(d)$is not decreasing
with $d$ increasing.
\end{remark}

We prove the Conjecture for any $d=q(d)^2+q(d)+1$, see Corollary \ref{co: q}.

The last assertion is a consequence of the following more general result.
\begin{theorem}[Lower bound on Harbourne constants]\label{thm: lower bnd}
For $d\geq 6$ we have
\[
H(d)\geq -\frac 1 2 \sqrt{4d-3}+\frac 1 2.
\]
For $d=q^2+q+1$ with $q\in Q$ we have the equality. In this case $H(d)=-q$ is computed by
the configuration consisting of all lines in the finite projective plane $\P^2(\F_q)$.
\end{theorem}

\begin{theorem}[Upper bound for Harbourne constants]\label{thm: upper bnd}
For $d\geq 7$ and with $r=r(d)$, we have
\[
H(d)\leq -2\frac{r^4+r^3-r-(d-1)^2}{r^4+2r^3-r-d^2+d-2}.
\]
\end{theorem}

   We will considerably improve this bound for some $d$ in Proposition \ref{prop: construction}.
   In order to prove Theorem \ref{thm values} we introduce some new tools, which might be of independent interest in other areas of combinatorics and geometry. 
   We discuss also how our problem is related to the classical geometric problem on the existence of projective planes with certain numbers of points. 
   Our investigations are accompanied by Singular \cite{DGPS} computations. The complete script of our program
   is provided in the Appendix.

\section{Initial data}

Let $\call=\{L_1,\ldots,L_d\}$ be a configuration of lines in the
projective plane $\P^2(\K)$. Let $t_k$ be the number of points where
exactly $k$ lines intersect. Then we have the following basic
combinatorial equality
\begin{equation}\label{eq: combinatorial}
\binom{d}{2}=\sum\limits_{k= 2}^{d}t_k\binom{k}{2}.
\end{equation}
Note that using this notation and taking into account \eqref{eq: combinatorial} we can simplify the way $H(\K,\call)$ is expressed:
\begin{equation}\label{eq: simply H}
H(\K,\call)=\frac{d^2-\sum_{k=2}^d t_kk^2}{s}=\frac{d-\sum_{k=1}^sm_{\call}(P_k)}{s}.
\end{equation}

Now our approach to computing or bounding Harbourne constants is
based on the following idea. For a fixed $d$ we consider the set
$\calt$ of all integral solutions $T=(t_2,t_3,\ldots,t_d)$ of the
equality \eqref{eq: combinatorial} and we compute the resulting
combinatorial quotient
\begin{equation}\label{eqqT}
   q(T)=\frac{d^2-\sum_{k=2}^dt_kk^2}{\sum_{k=2}^dt_k}.
\end{equation}
Of course not all elements of $\calt$ come from geometric configurations. So the task is to sort out those which cannot be obtained geometrically and then to find the minimum of $q(T)$'s for those which can.

\section{Criteria for the nonexistence of a geometric configuration}

In \cite{Szp2015} we introduced a number of criteria to deal with this problem. Here we begin with a useful modification of what was called a two pencils criterion. We keep this name and hope that this will not lead to any confusion.
\begin{lemma}[Two pencils criterion]\label{le: TP criterion}
Let $\call=\{L_1,\ldots,L_d\}$ be a configuration of lines in the
projective plane $\P^2(\K)$, with the singular set
$\{P_1,\ldots,P_s\}$, with $s\geq 2$. Let $m_1,\ldots,m_s$ be the
multiplicities of points $P_1,\ldots,P_s$ respectively. Without loss
of generality we can assume that
\[
m_1\geq m_2\geq\ldots\geq m_s.
\]
Then either
\begin{equation}\label{TPPv1}
m_1 m_2 + 2 \leq s,
\end{equation}
or, if \eqref{TPPv1} does not hold,
$$
(m_1-1)(m_2-1)+a\leq s,
$$
where $a$ is equal to the minimal number of singular points lying on
a line passing through $P_1$ and $P_2$. The number $a$ can be easily
computed combinatorially.
\end{lemma}

\begin{proof}
Assume that points $P_1$ and $P_2$ do not lie on a configuration lines, then the lines from two
pencils (lines through $P_1$, resp. $P_2$) meet in $m_1 m_2$ points. Together with $P_1$ and $P_2$ we get \eqref{TPPv1}.

If \eqref{TPPv1} does not hold, points $P_1$ and $P_2$ lie on a configuration line. Lines from these two pencils (apart from the common line) meet in $(m_1-1)(m_2-1)$ points. Now we add the number of points on the common line, which is at least $a$.
\end{proof}
The following Example illustrates how the two pencil criterion is
applied
\begin{example}
   The following data: $d=10$, $t_3=7$ and $t_4=4$ is a solution of \eqref{eq: combinatorial}. 
   Then $m_1=m_2=4$ and $s=11$.
   Since $4 \cdot 4 + 2 > 11$, we pass to the second inequality. Now
   $a=3$, since the line through $P_1$ and $P_2$ meets with six other
   lines at these two points, hence there must be another point on this
   line (and one point of multiplicity 4 suffices). The inequality 
   $3\cdot 3 + 3 > 11$ shows that there is no geometrical configuration satisfying above data.
\end{example}
   The next idea is to doubly count the incidences. First we need to introduce some notation. To a configuration line $L$ we attach its \emph{type vector}
\[
\nu(L)=(\nu_2(L),\nu_3(L),\ldots,\nu_d(L)),
\]
   where $\nu_k(L)$ denotes the number of points of multiplicity $k$ on $L$. For example the line $L$ in Figure \ref{fig: type of line} has type $\nu(L)=(1,2,0,0,0)$.

\begin{figure}[H]
\centering
\begin{tikzpicture}[line cap=round,line join=round,>=triangle 45,x=1.0cm,y=1.0cm,scale=0.5]
\clip(-4.3,-1.92) rectangle (7.0600000000000005,6.3);
\draw [domain=-4.3:7.0600000000000005,thick] plot(\x,{(--8.684800000000001--0.040000000000000036*\x)/3.4800000000000004});
\draw (5.3,3) node {$L$};
\draw [domain=-4.3:7.0600000000000005] plot(\x,{(--8.0032--1.1800000000000002*\x)/2.58});
\draw [domain=-4.3:7.0600000000000005] plot(\x,{(--4.752-1.32*\x)/2.64});
\draw [domain=-4.3:7.0600000000000005] plot(\x,{(--4.684800000000001-1.1400000000000001*\x)/0.9000000000000001});
\draw [domain=-4.3:7.0600000000000005] plot(\x,{(-0.7664--1.36*\x)/0.8400000000000001});
\draw [domain=-4.3:7.0600000000000005] plot(\x,{(--3.2696000000000005-2.5*\x)/0.06000000000000005});
\begin{scriptsize}
\draw [fill=black] (-1.36,2.48) circle (1.5pt);
\draw [fill=black] (2.12,2.52) circle (1.5pt);
\draw [fill=black] (1.22,3.66) circle (1.5pt);
\draw [fill=black] (1.28,1.16) circle (1.5pt);
\draw [fill=black] (3.455672068636797,4.682516682554813) circle (1.5pt);
\draw [fill=black] (4.441739130434785,-0.4208695652173919) circle (1.5pt);
\draw [fill=black] (1.2476006618863762,2.5099724214009926) circle (1.5pt);
\end{scriptsize}
\end{tikzpicture}
   \caption{}
   \label{fig: type of line}
   \end{figure}

Now, let $n_{\nu}(\call)$ be the number of lines in $\call$ with the type vector $\nu$. Then we have
 \begin{equation}\label{eq: solver}
   \left\{\begin{array}{ccc}
   &\sum\limits_{\nu} n_{\nu}(\call)=d\\
   &\sum\limits_{\nu=(\nu_2,\ldots,\nu_d)} n_{\nu}(\call)\cdot \nu_k=k\cdot t_k, \quad {\text{for $k=2,\ldots,d$.}}
   \end{array}\right.
 \end{equation}
The first equation simply counts all lines in a configuration. The others count all ``incidences'' --- a line passing through
a point of given multiplicity $k$ count as one incidence.

Let $T=\{t_2,\ldots,t_d\}$ be a set of integers satisfying
\eqref{eq: combinatorial} for a fixed $d$. To these numbers there is
the associated system of equations \eqref{eq: solver}. In this
system the symbols $n_{\nu}(\call)$ are unknown. Which type vectors
$\nu(L)$ can appear in the given configuration can be easily
determined in advance. Their number is quite restricted. If the
system \eqref{eq: solver} has no-negative solutions, then it follows
that the set $T$ cannot be realized geometrically.

However the set of equalities \eqref{eq: solver} is not always
sufficient for our purposes. Let us assume that in a configuration
there is a unique point of multiplicity $m$, that is $t_m=1$. Then
all lines passing through this point (that is all lines $L$ with
$\nu(L)=(\nu_2,\dots,\nu_{m-1},1,\nu_{m+1},\dots,\nu_d)$) belong to
the same pencil. Now for $k\neq m$ we count all points of
multiplicity $k$ on these lines, which obviously must be at most
$t_k$,
\begin{equation}\label{eqonep}
\sum\limits_{\nu=(\nu_2,\ldots,\nu_d), \nu_m=1} n_{\nu}(\call)\cdot \nu_k \leq t_k, \quad \text{for $k=2,\ldots,d$, $k\neq m$.}
\end{equation}
So the new and powerful criterion for nonexistence works as follows:
write down all equations \eqref{eq: solver} with the set of
inequalities \eqref{eqonep} for all $t_m=1$, then try to solve this
system of linear equations and inequalities in non-negative
integers. This is a problem, well-known as integer programming, and
there are many algorithms and software to deal with it.
\begin{example}\label{ex: d14}
   The following data: $d=14$, $t_3=7$, $t_4=10$ and $t_5=1$
   is a solution of \eqref{eq: combinatorial}. If it corresponds
   to a geometrical configuration then there are exactly four type vectors $\nu$, for
   which $n_\nu(\call)$ may be non-zero (a fixed line must meet with 13
   other in singular points using only multiplicities appearing in the
   configuration, hence we can easily write down all possibilities).
   These are $(0,5,1,0,\dots)$, $(0,2,3,0,\dots)$, $(0,3,1,1,\dots)$,
   $(0,0,3,1,\dots)$. Assume that we have $a$ (resp. $b$, $c$, $d$)
   lines in $\call$ with resp. types. We have the following system of
   equalities:
$$
\begin{cases}
a+b+c+d=14, \\
5a+2b+3c=21, \\
a+3b+c+3d=40, \\
c+d=5.
\end{cases}
$$
   There are two nonnegative integer solutions, namely $(a,b,c,d)\in \{(0,9,1,4),(1,8,0,5)\}$. 
   Observe however that $t_5=1$ allows us to use two additional inequalities
$$
\begin{cases}
3c \leq 7,\\
c+3d \leq 10.
\end{cases}
$$
   The last inequality gives a contradiction, hence the initial data in this example does
   not come from any geometrical configuration.
\end{example}

\subsection{A {\sc Singular} script}

We wrote a {\sc Singular} script, which, given a number of lines $d$
and a value $h:= q(T)$ for a geometric configuration, works as
follows:
\begin{itemize}
\item
it enumerates first all possible arrays of integers
$T=(t_2,\dots,t_d)$ satisfying \eqref{eq: combinatorial},
\item
for each such an array it computes the quotient $q(T)$ as in
\eqref{eqqT},
\item
for those quotients, which satisfy $q(T) < h$ it checks whether the
two pencils criterion works,
\item
if this is not the case, then the script produces an input for
linear programming problem given by \eqref{eq: solver} and
\eqref{eqonep}, then it uses a glpsol software to solve it,
\item
the results are reported; if for all $T$ with $q(T)<h$ one of the
two above criteria verifies the non-existence of a geometric
configuration with data $T$, then $h$ is a lower bound for $H(d)$.
Otherwise the test fails and we do not know $H(d)$
\end{itemize}
   The script is revoked by the command
   \begin{center}
   \begin{verbatim}
      check(number_of_lines, tested_bound, "output_file");
   \end{verbatim}
   \end{center}
   For example check(10,-29/12,"result") checks the validity of the number $H(10)$
   provided in Theorem \ref{thm values}.

\section{Proofs of the lower and upper bounds}
In this section we prove Theorem \ref{thm: lower bnd} and Theorem \ref{thm: upper bnd}. We begin with the lower bound.
\proofof{Theorem \ref{thm: lower bnd}}
For $d\geq 6$ and $s\geq 1$ we consider the following function
$$f(d,s)= \frac d s -\frac 1 2 -\frac 1 2 \sqrt{1+\frac{4d^2-4d}{s}}.$$ For a fixed field $\K$ and positive integer $d$, let
$\call$ be a configuration of $d$ lines with altogether $s$ singular
points. Then we have the following\\
\textbf{Claim}
   \begin{equation}\label{ineq:lower}
    H(\K, \call) \geq f(d,s).
   \end{equation}
   Taking this for granted, Theorem \ref{thm: lower bnd} follows easily. Indeed, first of
   all the right hand side in (\ref{ineq:lower}) does not depend on
   $\K$, so that
   $$H(d)\geq \min\limits_{s\geq 1} f(d,s).$$
If $\call$ is a pencil, i.e. $s=1$, then $H(\K,\call)=f(d,1)=0$.
Otherwise by the celebrated de Bruijn-Erd\"os Theorem \cite{BruErd48}
it must be $s\geq d$. Elementary calculus shows that for a fixed $d$
the function $f(d,s)$ is strictly increasing for $s \geq d$. Hence
finally
\[
H(d)\geq \min\limits_{s\geq d} f(d,s)= -\frac 1 2 \sqrt{4d-3}+\frac
1 2.
\]
The extra assertion of Theorem \ref{thm: lower bnd} will be proved at the end of this section.

Now we turn back to the Claim.

Using \eqref{eq: simply H} we have
\[
H(\K, \call)=\frac d s -\frac{\sum_{k=1}^sm_{\call}(P_k)}{s},
\]
so it suffices to show that
\begin{equation}\label{ieq: M}
M:=\frac{\sum_{k=1}^sm_{\call}(P_k)}{s}\leq \frac 1 2+\frac 1 2 \sqrt{1+\frac{4d^2-4d}{s}}.
\end{equation}
The idea now is to apply Jensen's inequality \eqref{ieq: Jensen} to \eqref{eq: combinatorial}.

Recall that for a convex function $\varphi(x)$ and non-negative numbers $\lambda_1,\ldots,\lambda_s$ such that $\sum_{i=1}^s\lambda_i=1$ there is
\begin{equation}\label{ieq: Jensen}
\sum_{i=1}^s\lambda_i\varphi(x_i)\geq \varphi\left(\sum_{i=1}^s\lambda_ix_i\right).
\end{equation}
The function $\varphi(x)=x(x-1)$ satisfies the assumptions. Hence, from \eqref{eq: combinatorial} we obtain with $\lambda_1= \ldots =\lambda_s=\frac 1 s$
   \begin{equation}
      \renewcommand{\arraystretch}{1.2}
   \begin{array}{ccl}
       \frac 1 s d(d-1) & = & \frac 1 s \sum_{k=1}^s m_{\call}(P_k)(m_{\call}(P_k)-1)\geq\\
       & \geq & (\frac 1 s \sum_{k=1}^s m_{\call}(P_k))(\frac 1 s \sum_{k=1}^s m_{\call}(P_k)-1)  \\
       & = & M(M-1).
   \end{array}
   \end{equation}
It is elementary to check that this implies \eqref{ieq: M} and we
are done.
\endproof
Now we prove the upper bound.
\proofof{Theorem \ref{thm: upper bnd}}
This bound is obtained in a rather naive way. Let $r=r(d)$. We consider the projective plane $\P^2(\F_{r})$ as embedded in the projective plane defined over the algebraic closure $\overline{\F}_{r}$. Then $\call_1$ is the configuration of all $d_1=r^2+r+1$ lines coming from $\P^2(\F_{r})$. Then we take $d_2=d-d_1$ \emph{general} lines in $\P^2(\overline{\F}_{r})$. These lines form another configuration $\call_2$. Since they are general, they intersect pairwise in $\binom{d_2}{2}$ distinct points and they intersect the lines in $\call_1$ in $d_2d_1$ distinct points. Thus for $\call=\call_1\cup\call_2$ we have
   \begin{equation}
   t_k(\call)=
   \left\{\begin{array}{lcc}
   \frac 1 2 d_2(d_2-1)+d_1d_2 &   & \text{for $k=2$}\\
   r^2+r+1 &   & \text{for $k=r+1$}\\
   0 &   & \text{otherwise}
   \end{array}\right..
   \end{equation}
The bound follows then computing the Harbourne constants and expressing everything in terms of $d$ and $r=r(d)$. Note that for $d=r^2+r+1$, we get $H(d)\leq -r(d)$.
\endproof
We conclude this section by showing the extra claim in Theorem \ref{thm: lower bnd}.
\begin{corollary}
Let $d=q^2+q+1$ with $q\in Q$. Then
\[
H(d)=-q
\]
and $H(d)$ is computed by the configuration of all lines in $\P^2(\F_q)$.
\end{corollary}\label{co: q}
\begin{proof}
   The inequality $H(d)\leq -q$ follows from Theorem \ref{thm: upper bnd}. 
   The lower bound $H(d)\geq -q$ follows in turn from Theorem \ref{thm: lower bnd}. 
   Note additionally that the proof of Theorem \ref{thm: lower bnd} shows then that there is the equality in \eqref{ieq: M}. 
   Hence $d=s$ in this case and we conclude again by the de Bruijn-Erd\"os Theorem. Note that whereas the configuration
   consists of all
   all lines in $\P^2(\F_q)$, it might be embedded in some larger projective plane.
\end{proof}

\section{Results justifying Conjecture \ref{con: Hconstants}}
We show first that there are infinitely many values of $d$ such that
$$H(d)\leq h(d)$$
holds. More precisely we have the following result
\begin{proposition}\label{prop: construction}
Let $d$ be a positive integer such that
$$(q-1)^2+(q-1)+1< d\leq q^2+q+1$$
with $q=q(d)$. Then
$$H(d)\leq h(d).$$
\end{proposition}
\begin{proof}
The idea is to construct a configuration of lines with invariants indicated in Conjecture \ref{con: Hconstants}. To this end let
$i=q^2+q+1-d$. Let $\call_0$ be the configuration of all lines in $\P^2(\F_q)$. For $i\leq q+1$, we fix a point $P_1\in \P^2(\F_q)$
and remove exactly $i$ lines passing through the point $P_1$ getting the configuration $\call_1$. These lines intersect only at $P_1$, so that with every line we decrease the number of $(q+1)$-fold points by $q$ and increase the number of $q$-fold points by $q$ as well. The multiplicity $m_1$ of the point $P_1$ is $q+1-i$, whereas for $i=q$ and $i=q+1$, the point is no more a singular point
of the configuration. It is then elementary to check that
\begin{equation}\label{eq: Hbound1}
H(\F_q,\call_1)=h(d).
\end{equation}
For $q+1<i\leq 2q-2$, we remove all $q+1$ lines passing through $P_1$. This results in a configuration of $q^2$ lines with
$q^2+q$ points of multiplicity $q$. Then we remove the remaining $i-(q+1)$ lines from the pencil of lines passing through a second point $P_2$. Counting as above, we get \eqref{eq: Hbound1}.

Finally for $i=2q-1$, we fix two points $P_1$, $P_2$ and remove the line joining them, and $2(q-1)$ additional lines:
$q-1$ from a pencil through $P_1$ and $q-1$ from the other pencil. This results in a configuration of $q^2-q+2$ lines with
   \begin{equation}
      \renewcommand{\arraystretch}{1.2}
   \begin{array}{ccl}
       t_{q+1} & = & 1\\
       t_q & = & 3(q-1) \\
       t_{q-1} & = & (q-1)^2
   \end{array}
   \end{equation}
and all other $t_k=0$. This gives $H(d)\leq h(d)$ also in this case.
\end{proof}
Now we are in the position to prove Theorem \ref{thm values}.
\proofof{Theorem \ref{thm values}}
For all $d$ between $2$ and $31$, Proposition \ref{prop: construction} applies, so that $H(d)$ is at most equal to the numbers stated in Table \ref{tab: values}. Turning to the lower bound it turns out that our {\sc Singular} script works in all cases. This ends the proof.
\endproof
We pass now to $d$ in the range $32\leq d \leq 43$. For these values of $d$ we have $q=q(d)=7$, so consequently $d\leq (q-1)^2+(q-1)+1$. Hence the construction used in Proposition \ref{prop: construction} does not apply. It is well known that there is no projective plane with $43$ points (this would correspond to $q=6$), see \cite{BruRys49}.

Of course $t_7=43$ and all other $t_k=0$ is a solution to \eqref{eq: combinatorial} with $d=43$. Our program cannot exclude this configuration. It also cannot exclude any configuration resulting from this fake $\P^2(\F_6)$ configuration by removing lines. What it can is to exclude any lower values of Harbourne constants. So that we can conclude that for $32\leq d\leq 42$ it is
$$H(d)\geq h(d)$$
and
$$H(43)>-6.$$
It would be very interesting to modify our approach in a way opening access to configurations coming from fake projective planes. We hope to come back to this question in the next future.

\paragraph*{\emph{Acknowledgement.}}
   This paper was to large extend written while the last author visited
   the University of Freiburg. It is a pleasure to thank Stefan Kebekus
   for hospitality.
   We would like also to thank Tomasz Szemberg for helpful conversations.

\section{Appendix}
\begin{verbatim}
ring R=0,x,dp;
proc writelistH(list l) {
  string s="H-const: "+string(l[1])+"  conf: ";
  for (int k=2;k<=size(l);k++) {
    s=s+string(l[k])+" ";
  }
  return(s);
}
proc writelist(list l) {
  string s="";
  for (int k=2;k<=size(l);k++) {
    s=s+string(l[k])+" ";
  }
  return(s);
}
proc computeH(int n, list m) {
  number h=n*n;
  for (int i=2;i<=size(m);i++) {
    h=h-i*i*m[i];
  }
  int s=0;
  for (i=2;i<=size(m);i++) {
    s=s+m[i];
  }
  h=h/s;
  return(h);
}
proc twopencil(int n, list m) {
  int s=0;
  for (int i=2;i<=size(m);i++) {
    s=s+m[i];
  }
  string result="";
  if (s<n) {return("  few points works");}
  int r=2;
  int p=size(m);
  list tp;
  while (r>0) {
    if (m[p]>0) {
      tp[r]=p;
      m[p]=m[p]-1;
      r=r-1;
    } else {
      p=p-1;
    }
  }
  if ((tp[1]-1)*(tp[2]-1)+2>s) {result=result+ "  TP works";} else
  {
  int dp=2;
  int zp=n-(tp[1]-1)-(tp[2]-1)-1;
  while (zp>0) {
    if (m[p]>0) {
      zp=zp-(p-1);
      m[p]=m[p]-1;
      dp=dp+1;
    }
    else
    {
      p=p-1;
    }
  }
  if ((tp[1]-1)*(tp[2]-1)+dp>s) {result=result+"  TP(p) works";}
  }
  return(result);
}
proc writeconf(list l) {
  string s="   line conf: ";
  for (int k=2;k<=size(l);k++) {
    s=s+string(l[k])+" ";
  }
  return(s);
}
proc eqcritbyglp(int n, list m) {
  "checking "+writelist(m);
  "  conf for n="+string(n)+"...";
  list rm,nm,sm;
  for (int i=1;i<=size(m);i++) {rm[i]=0;}
  int p=0;
  int s;
  list cp;
  while (p<=size(m)) {
    rm[2]=rm[2]+1;
    p=2;
    while ((p<=size(m))&&(rm[p]>m[p])) {
      rm[p]=0;
      p=p+1;
      if (p<=size(m)) {rm[p]=rm[p]+1;} else {break;}
    }
    s=0;
    for (i=2;i<=size(m);i++) {
      s=s+rm[i]*(i-1);
    }
    if (s==n-1) {
      cp[size(cp)+1]=rm;
    }
  }
  int vrb=size(cp);
  if (vrb==0) {return("   CONF(0) works");}
  list eqs;
  list eqq;
  for (i=1;i<=size(cp);i++) {
    eqq[i]=1;
  }
  eqq[vrb+1]=n;
  eqs[1]=eqq;
  for (p=2;p<=size(m);p++) {
    if (m[p]>0) {
      for (i=1;i<=size(cp);i++) {
        rm=cp[i];
        eqq[i]=rm[p];
      }
      eqq[vrb+1]=p*m[p];
      eqs[size(eqs)+1]=eqq;
    }
  }
  int j;
  string name=":w test";
  write(name,"minimize value: a1");
  name=":a test";
  write(name,"subject to");
  string wr;
  for (i=1;i<=size(eqs);i++) {
    wr="e"+string(i)+": ";
    eqq=eqs[i];
    for (j=1;j<=vrb;j++) {
      if (j>1) {wr=wr+" + ";}
      wr=wr+string(eqq[j])+" a"+string(j);
    }
    wr=wr+" = "+string(eqq[vrb+1]);
    write(name,wr);
  }
  int mm,k,o;
  for (mm=2;mm<=size(m);mm++) {
    if (m[mm]==1) {
      for (i=2;i<=size(m);i++) {
        if ((i!=mm)&&(m[i]>0)) {
          wr="b"+string(mm)+"k"+string(i)+": ";
          o=0;
          for (k=1;k<=size(cp);k++) {
            rm=cp[k];
            if (rm[mm]==1) {
              if (o==0) {o=1;} else {wr=wr+" + ";}
              wr=wr+string(rm[i])+" a"+string(k);
            }
          }
          wr=wr+" <= "+string(m[i]);
          if (o==1) {write(name,wr);}
        }
      }
    }
  }
  write(name,"integer");
  for (j=1;j<=vrb;j++) {
    write(name," a"+string(j));
  }
  write(name,"end");
  int dummy=system("sh","glpsol --lp test -o solution");
  link solfile=":r solution";
  string sol=read(solfile);
  if (find(sol,"UNDEFINED",1)>10) {return("  SOLVER works");}
  return("");
}
proc throw(int n, list m) {
  string ii=writelistH(m);
  ii=ii+twopencil(n,m);
  if (find(ii,"works")==0) {ii=ii+eqcritbyglp(n,m);}
  return(ii);
}
proc scheck(int n, number bnd, string infofile, list m) {
  int ntp,ntpp,nsolver;
  int ok=1;
  infofile=":a "+infofile;
  write(infofile,"Input data: "+string(n)+" lines, bound for H-constant: 
     "+string(bnd)+".");
  write(infofile,"Configurations to exclude:");
  list v;
  string info;
  int na;
  for (int i=2;i<=n;i++) {
    v[i]=i*(i-1) div 2;
  }
  int sum=n*(n-1) div 2;
  list b;
  for (i=2;i<=n;i++) {
    b[i]=sum div v[i];
  }
  int p;
  int mm;
  while (p<n) {
    m[3]=m[3]+1;
    p=3;
    while ((m[p]>b[p])||((p>n+1-mm)&&(p<mm))) {
      m[p]=0;
      p=p+1;
      m[p]=m[p]+1;
      if (p>mm) {mm=p;}
    }
    if (p==n) {break;}
    m[2]=sum;
    for (i=3;i<=n-1;i++) {
      m[2]=m[2]-v[i]*m[i];
    }
    if ((m[2]>=0)&&(m[n]==0)) {
      na=na+1;
      if ((na mod 10000)==0) {string(na)+" already checked...";writelist(m);}
      m[1]=computeH(n,m);
      if (m[1]<bnd) {
        info=throw(n,m);
        info;
        if (find(info,"works")==0) {ok=0;}
        write(infofile,info);
        if (find(info,"TP ")>0) {ntp++;}
        if (find(info,"TP(p)")>0) {ntpp++;}
        if (find(info,"SOLVER")>0) {nsolver++;}
      }
    } else {
      p=3;
      while (m[p]==0) {p=p+1;}
      m[p]=b[p];
    }
  }
  if (ok==1) {
    "All configurations have been excluded.";
    write(infofile,"All configurations have been excluded.");
  }
  info="TP used "+string(ntp)+" times, TP(p) used "+string(ntpp)+" times, 
     SOLVER used "+string(nsolver)+" times.";
  info;
  write(infofile,info);
}
proc check(int n, number bnd, string infofile) {
  list m;
  for (int i=1;i<=n;i++) {
    m[i]=0;
  }
  scheck(n,bnd,infofile,m);
}
proc contcheck(int n, number bnd, string infofile, list sm) {
  list m;
  for (int i=1;i<=n;i++) {
    m[i]=0;
  }
  for (i=1;i<=size(sm);i++) {
    m[i]=sm[i];
  }
  scheck(n,bnd,infofile,m);
}
\end{verbatim}


\bigskip \small
\bigskip
   Marcin Dumnicki,
   Jagiellonian University, Institute of Mathematics, {\L}ojasiewicza 6, PL-30-348 Krak\'ow, Poland

\nopagebreak
   \textit{E-mail address:} \texttt{Marcin.Dumnicki@im.uj.edu.pl}

\bigskip
   Daniel Harrer,
   Albert-Ludwigs-Universit\"at Freiburg,
   Mathematisches Institut, D-79104 Freiburg, Germany

\nopagebreak

   \textit{E-mail address:} \texttt{daniel.harrer@math.uni-freiburg.de}

\bigskip
   Justyna Szpond,
   Pedagogical University of Cracow, Institute of Mathematics,,
   Podchor\c a\.zych 2,
   PL-30-084 Krak\'ow, Poland

\nopagebreak

   \textit{E-mail address:} \texttt{szpond@up.krakow.pl}

\bigskip

   Justyna Szpond current address:
   Albert-Ludwigs-Universit\"at Freiburg,
   Mathematisches Institut, D-79104 Freiburg, Germany

\end{document}